\documentclass[12pt,a4paper]{amsart}
\usepackage{amsfonts}
\usepackage{amsthm}
\usepackage{amsmath}
\usepackage{amssymb}
\usepackage{amscd}
\usepackage{t1enc}
\usepackage[margin=2.9cm]{geometry}
\usepackage{epstopdf} 
\usepackage{times}
\usepackage{enumerate}
\usepackage{changepage}
\usepackage{mfirstuc}  
\usepackage{tikz}
\usepackage{cite}
\usepackage{booktabs}
\usepackage{tabularx}
\numberwithin{equation}{section}

\theoremstyle{plain}
\newtheorem{Th}{Theorem}[section]
\newtheorem{Lemma}[Th]{Lemma}
\newtheorem{Conj}[Th]{Conjecture}
\newtheorem{Cor}[Th]{Corollary}
\newtheorem{Prop}[Th]{Proposition}
\newtheorem{Def}[Th]{Definition}

\newcommand{\Q}{\mathbb{Q}}
\newcommand{\R}{\mathbb{R}}
\newcommand{\Z}{\mathbb{Z}}
\newcommand{\F}{\mathbb{F}}

\newcommand{\LF}{\mathrm{L}}
\newcommand{\V}{\mathrm{V}}

\newcommand{\ord}{\operatorname{ord}}
\newcommand{\Vol}{\operatorname{Vol}}
\newcommand{\Tr}{\operatorname{Tr}}

\begin{document}

\title{L-functions of Certain Exponential Sums over Finite Fields}

\author{Chao Chen}
\address{Department of Mathematics, University of California, Irvine, 
Irvine, CA 92697-3875, USA}
\email{chaoc12@uci.edu}

\author{Xin Lin}
\address{Department of Mathematics, Shanghai Maritime University, Shanghai 201306, PR China}
\email{xlin1126@hotmail.com}
\thanks{X. Lin is supported by China Scholarship Council(No. 201906970038) and National Natural Science Foundation of China(No. 11771351, 11871317).}

\subjclass[2010]{Primary 11S40, 11T23, 11L07}

\keywords{Exponential sums, L-function, Laurent polynomials, Newton polygon, Hodge polygon, Decomposition theory, Weight computation}

\begin{abstract}In this paper, we completely determine the slopes and weights of the L-functions of an important class of exponential sums arising from analytic number theory. Our main tools include Adolphson-Sperber's work on toric exponential sums and Wan's decomposition theorems. One consequence of our main result is a sharp estimate of these exponential sums. Another consequence is to obtain an explicit counterexample of  Adolphson-Sperber's conjecture on weights of toric exponential sums.
\end{abstract}

\maketitle

\section{Introduction}\label{sec1}
Let $\F_{q}$ be the finite field of $q$ elements with characteristic $p$. For each positive integer $k$, let $\F_{q^k}$ denote the degree $k$ finite extension of $\F_{q}$ and $\F_{q^k}^*$ denote the set of non-zero elements in $\F_{q}$. Let $\psi: \F_{p} \rightarrow \mathbb{C}^*$ be a fixed nontrivial additive character over $\F_{p}$. In this paper, we are concerned with estimating the following exponential sum
\begin{align*}
S_k(\vec{a}) =S_k(a_1,a_2, \ldots,a_6)=\sum_{\substack{{\frac{a_5}{x_1x_2}+\frac{a_6}{x_3x_4}=1}\\x_i\in \F^*_{q^k}}}\psi (\Tr_k(a_1x_1+a_2x_2+a_3x_3+a_4x_4)),
\end{align*}
where $\Tr_k: \F_{q^k} \rightarrow \F_{p}$ is the trace map and $a_i\in \F_{q}^*,\ i=1,2,\ldots,6$.

In the past few decades, there have been extensive study and application of the exponential sum $S_k(\vec{a})$. For instance, using Deligne's theorem on weights, Birch and Bombieri proved that if $p>c_0$, then $|S_k(\vec{a})|\leq c_1q^{3k/2}$ for some absolute constants $c_0$,$c_1$ \cite{birch1985some}. Their estimate is a crucial ingredient in Heath Brown's work on the divisor function $d_3(n)$ in arithmetic progressions\cite{heath1986divisor}. Birch and Bombieri's result also plays a vital role in Friedlander and Iwaniec's work on estimating certain averages of incomplete Kloosterman sums in application to the divisor problem of $d_3(n)$\cite{friedlander1985incomplete}. Relying on Friedlander-Iwaniec's and Birch-Bombieri's results of $S_k(\vec{a})$, Zhang gained a boundary of the error terms in his work on twin prime conjecture\cite{zhang2014bounded}. A consequence of our theorem is a sharp estimate: for all $p,k$ and $\vec{a}$, 
\begin{equation*}
|S_k(\vec{a})|\leq6q^{3k/2}+q^{k}+1.
\end{equation*}

Let
\begin{align}\label{eq11}
\LF(\vec{a},T)=\exp  \left(\sum^\infty _ {k=1} S_k(\vec{a}) \frac{T^k}{k}\right).
\end{align}
Our main result in this paper is a complete determination of the weights and slopes of the L-function attached to the sequence of exponential sums $S_k(\vec{a})$ ($k=1,2,\ldots$). 
\begin{Th}\label{thm1}
	The generating L-function of  $S_k(\vec{a})$ is a polynomial of degree 8, i.e.,
	\begin{equation*}
	\LF(\vec{a},T)=(1-T)(1-qT)\prod_{i=1}^6(1-\alpha_iT),
	\end{equation*}
	where $\alpha_i \in \overline{\Q}$ with the complex absolute value $|\alpha_i|=q^{3/2}$.	
	If we view the reciprocal roots $\alpha_i$ as $p$-adic numbers and enumerate them with respect to the $q$-adic slopes, their $p$-adic norms are given by
	\begin{align*}
		|\alpha_i|_p=
		\begin{cases}
			1 &\text{if}\quad i =1.\\
			q^{-1} &\text{if}\quad i= 2,3.\\
			q^{-2} &\text{if}\quad i= 4,5.\\
			q^{-3} &\text{if}\quad i= 6.
		\end{cases}
	\end{align*}
\end{Th}
\begin{Cor}
Notations as above. For all $p,k$ and $\vec{a}$, 
\begin{equation*}
|S_k(\vec{a})|\leq6q^{3k/2}+q^k+1.
\end{equation*}
\end{Cor}

Our approach is to reduce this theorem to the L-function of toric exponential sums 
and then apply the systematic results available for such toric L-functions. Consider a Laurent polynomial $g\in \F_{q}[x_1^{\pm1},\ldots, x_5^{\pm1}]$ defined by
\begin{equation*}
g(x_1,x_2,\ldots,x_5)=a_1x_1+a_2x_2+a_3x_3+a_4x_4+x_5\left(\frac{a_5}{x_1x_2}+\frac{a_6}{x_3x_4}-1\right),
\end{equation*}
where $a_i \in \F^*_{q}$ ($i=1,2, \ldots, 6$). For any positive integer $k$, let 
 \begin{equation*}
 S^*_k(g)=\sum_{x_i\in \F^*_{q^k}}\psi(\Tr_k(g))
 \end{equation*}
 be the associated exponential sum. Its generating L-function is defined to be
 \begin{equation*}
 \LF^*(g,T)=\exp  \left(\sum^\infty _ {k=1} S^*_k(g) \frac{T^k}{k}\right).
 \end{equation*}
With Adolphson and Sperber's theorem\cite{AS1989} and Wan's boundary decomposition theorem\cite{Dwan1993}, $\LF^*(g,T)$ is a polynomial of degree 9, i.e., 
\begin{equation*}
\LF^*(g,T)=(1-T)\prod_{i=1}^8(1-\beta_iT),
\end{equation*}
where $\beta_i \in \Z[\zeta_p]$ and $|\beta_i| \leq q^{5/2}$, $i=1,2,\ldots, 8$. 

The following equation describes the relationship between $S_k(\vec{a})$ and  $S^*_k(g)$,
 \begin{equation*}
 S_k(\vec{a})=\frac{1}{q^k}+\frac{1}{q^k}S^*_k(g).
 \end{equation*}
Based on the relationship, it's easy to check that
\begin{align}\label{eq12}
\LF(\vec{a},T)=\frac{1}{1-T/q}\LF^*\left(g,T/q\right)=\prod_{i=1}^8(1-(\beta_i/q)T).
\end{align}
So it suffices to evaluate the nontrivial reciprocal roots of $\LF^*(g,T)$ , denoted by $\beta_i(1\leq i \leq 8)$. To gain the slopes, we use Wan's facial decomposition theorem to compute the $q$-adic Newton polygon of $\LF^*(g,T)$. Wan's boundary decomposition theorem together with the slope information we obtained lead to an explicit expression of $\LF^*(g,T)$,
\begin{equation*}
\LF^*(g,T)=(1-T)(1-qT)(1-q^2T)\prod_{i=1}^6(1-\beta_iT),
\end{equation*}
where $\beta_i \in \Z[\zeta_p]$ and $|\beta_i| = q^{5/2}$, $i=1,2,\ldots, 6$. Our main theorem follows from this result.
For the weight computation of $\LF^*(g,T)$, one can also apply Denef-Loeser's weight formula obtained using intersection cohomology \cite{Denef1991}. This formula becomes combinatorially complicated to use when the dimension $n\geq 5$ and the polytope is not simple at the origin. This is precisely the case 
for our $5$-dimensional example $g$. In contrast to Wan's decomposition theorem, Denef-Loeser's formula can't determine the  two real roots. Moreover, if splitting the exponential sum $S^*_k(g)$ into a product of two Kloosterman sums, one can also obtain the real roots and weight information using Katz's calculation of the local monodromy of the Kloosterman sheaf, which is more advanced \cite{katz}. 

We remark that there is a much simpler weight formula conjectured by Adolphson-Sperber  \cite{Adolphson1990} which is true in many interesting cases including all low dimensional cases $n\leq 4$. This conjectural formula however was disproved by Denef and Loeser\cite{Denef1991} who showed the existence (no construction) of a counterexample in dimension $5$. 
We also test our example using the Adolphson-Sperber formula and found it disagrees with the result obtained using the Wan's decomposition theorem. This means that our $5$-dimensional example provides the first explicit construction of a counterexample to the Adophson-Sperber conjecture. 

This paper is organized as follows. In section \ref{sec2}, we review some technical methods and theorems including Adolphson-Sperber's theorems and Wan's decomposition theorems. In section \ref{sec3}, we prove the main results. In the appendix, we list the vertices of all faces of the polytope that are needed in both the slope computation and the weight computation. These were obtained via a simple computer calculation.

\section{Preliminaries}\label{sec2}
\subsection{Rationality of the generating L-function}\label{subsec21}
For a Laurent polynomial $f\in\F_{q}\left[x_1^{\pm1}, \ldots, x_n^{\pm1} \right]$, the associated exponential sum is defined by
\begin{align}\label{eq1}
S^*_k(f)=\sum_{x_i \in \F^*_{q^k}} \psi(\Tr_k(f)),
\end{align}
where $\Tr_k: \F_{q^k} \rightarrow \F_{p}$ is the trace map and $\psi: \F_{p} \rightarrow \mathbb{C}^*$ is a fixed nontrivial additive character. 
In analytic number theory, it's a classical problem to give a good estimate for the valuations of $S^*_k(f)$. In order to directly compute the absolute values of the exponential sums, we usually study the generating L-function of $S^*_k(f)$ defined as
\begin{equation*}
\LF^*(f,T)=\exp  \left(\sum^\infty _ {k=1} S^*_k (f) \frac{T^k}{k} \right) \in \Q(\zeta_p)[[T]].
\end{equation*}
By a theorem of Dwork-Bombieri-Grothendieck\cite{Dwork1962OnTZ, Grothendieck1964}, the generating L-function is a rational function given by
\begin{equation*}
\LF^*(f,T)=\frac{\prod^{d_1}_{i=1}(1-\alpha_iT)}{\prod^{d_2}_{j=1}(1-\beta_jT)},
\end{equation*}
where all the reciprocal roots and poles are non-zero algebraic integers. After taking logarithmic derivatives, we have the formula
\begin{align}\label{expo}
S_k^*(f)=\sum^{d_2}_{j=1}\beta_j^k-\sum^{d_1}_{i=1}\alpha_i^k, \quad k \in \Z_{\geq1}.
\end{align}
This formula implies that the zeros and poles of the generating L-function contain critical information about the exponential sums.

From Deligne's theorem on Riemann hypothesis\cite{Deligne1980}, the complex absolute values of reciprocal zeros and poles are bounded as follows
\begin{equation*}
|\alpha_i|=q^{u_i/2}, |\beta_j|=q^{v_j/2}, u_i\in \Z \cap [0,2n], v_j\in \Z \cap [0,2n].
\end{equation*}
For non-archimedean absolute values, Deligne\cite{Deligne1980} proved that $|\alpha_i|_\ell=|\beta_j|_\ell=1$ when $\ell$ is a prime and $\ell\neq p$. Depending on Deligne's integrality theorem\cite{Deligne1980}, we have the following estimates for $p$-adic absolute values
\begin{equation*}
|\alpha_i|_p=q^{-r_i}, |\beta_j|_p=q^{-s_j}, r_i\in \Q \cap [0,n], s_j\in \Q \cap [0,n].
\end{equation*}
The integer $u_i$ (resp. $v_j$) is called the \textit{weight} of $\alpha_i$ (resp. $\beta_j$) and the rational number $r_i$
 (resp. $s_j$) is called the \textit{slope} of $\alpha_i$ (resp. $\beta_j$). In the past few decades, it has been tremendous interest in determining the weights and slopes of the generating L-functions. Without any further condition on the Laurent polynomial $f$ or prime $p$, it's even hard to determine the number of reciprocal roots and poles. Adolphson and Sperber\cite{AS1989} proved that under a suitable smoothness condition of $f$ in $n$ variables, the associated L-function $\LF^*(f,T)^{(-1)^{n-1}}$ is a polynomial, i.e., $\LF^*(f,T)$ is a polynomial or the inverse of a polynomial. In this case, the Newton polygon can be used to determine the slopes of the reciprocal roots. Adolphson and Sperber\cite{AS1989} also proved that if $\LF^*(f,T)^{(-1)^{n-1}}$ is a polynomial, its Newton polygon has a lower bound called Hodge polygon. The basic definitions of the Newton polygon and the Hodge polygon will be discussed in the next subsection.

\subsection{Newton polygon and Hodge polygon}\label{subsec22}
Let 
\begin{equation*}
f(x_1, \ldots x_n)=\sum^J_{j=1} a_j x^{V_j}
\end{equation*}
 be a Laurent polynomial with $a_j \in \F ^*_{q}$ and $V_j=(v_{1j},\ldots, v_{nj})\in \Z^n$ $(1\leq j\leq J)$. The \emph{Newton polyhedron} of $f$, $\Delta(f)$, is defined to be the convex closure in $\R^n$ generated by the origin and the lattice points $V_j$ ($1\leq j \leq J$). For $\delta \subset \Delta(f)$, let the Laurent polynomial
\begin{equation*}
f^{\delta}=\sum_{V_j\in \delta}a_j x^{V_j}
\end{equation*}
be the restriction of $f$ to $\delta$. Then we define a certain smoothness condition of a Laurent polynomial.
 \begin{Def}
 	A Laurent polynomial $f$ is called non-degenerate if for each closed face $\delta$ of $\Delta(f)$ of arbitrary dimension which doesn't contain the origin, the $n$-th partial derivatives 
 	\begin{equation*}
 	\left\{ \frac{\partial f^{\delta}}{\partial x_1 }, \ldots, \frac{\partial f^{\delta}}{\partial x_n} \right\}
 	\end{equation*} 
 	have no common zeros with $x_1\ldots x_n \neq 0$ over the algebraic closure of $\F_q$.
 \end{Def}
 Generally, we are more interested in the non-degenerate Laurent polynomials whose generating L-functions are polynomials or the inverse of polynomials. 
\begin{Th}[Adolphson and Sperber\cite{AS1989}]\label{thm2}
	For any non-degenerate $f\in  \F_{q}[x_1^{\pm1},\ldots, x_n^{\pm1}]$, the associated L-function $\LF^*(f,T)^{(-1)^{n-1}}$ is of the following form, 
	\begin{equation*}
	\LF^*(f,T)^{(-1)^{n-1}}=\prod^{n! \Vol(\Delta(f))}_{i=1}(1-\alpha_i T),
	\end{equation*}
	where $|\alpha_i|=q^{\omega_i/2}$, $\omega_i\in \Z\cap [0,n]$ and $i=1,2,\ldots, n!\Vol(\Delta).$
\end{Th}
Deligne's integrality theorem\cite{Deligne1980} implies that the $p$-adic absolute values of reciprocal roots are given by $|\alpha_i|_p=q^{-r_i}$ where $r_i\in \Q\cap[0,n]$. For simplicity, we normalize $p$-adic absolute value to be $|q|_p=q^{-1}$. So we can compute the $q$-adic Newton polygon of $\LF^*(f,T)^{(-1)^{n-1}}$ to determine the $q$-adic slopes of its reciprocal roots. The $q$-adic Newton polygon is defined as follows.
\begin{Def}[Newton polygon]\label{m1}
	Let $\LF(T)=\sum^n_{i=0}a_i T^i$ $\in 1+T\overline{\Q}_p[T]$, where $\overline{\Q}_p$ is the algebraic closure of $\Q_p$. The $q$-adic Newton polygon of $\LF(T)$ is defined to be the lower convex closure of the set of points $\{\left(k,\ord_q(a_k)\right) | k=0, 1,\ldots, n \}$ in $\R^2$.
\end{Def}
Here $\ord_q$ denotes the standard $q$-adic ordinal on $\overline{\Q}_p$ where the valuation is normalized by assuming $\ord_q(q)=1$. The following lemma\cite{koblitz2012p} describes the relationship between the shape of a $q$-adic Newton polygon and the $q$-adic valuation of all the related reciprocal roots.
\begin{Lemma}\label{le4}
	In the above notation, let $\LF(T)=(1-{\alpha_1}T)\ldots(1-{\alpha_n}T)$ be the factorization of $\LF(T)$ in terms of reciprocal roots $\alpha_i \in \overline{\Q}_p$. Let $\lambda_i=\ord_q\alpha_i$. If $\lambda$ is the slope of a $q$-adic Newton polygon with horizontal length $l$, then precisely $l$ of the $\lambda_i$ are equal to $\lambda$.
\end{Lemma}
Assume $f$ is a non-degenerate Laurent polynomial in $n$ variables. So by Theorem \ref{thm2}, the L-function $\LF^*(f,T)^{(-1)^{n-1}}$ is a polynomial. Let NP($f$) denote the $q$-adic Newton polygon of $\LF^*(f,T)^{(-1)^{n-1}}$. It's always hard to directly compute NP($f$). Adolphson and Sperber proved that NP($f$) has a topological lower bound called Hodge polygon, which is easier to calculate\cite{AS1989}. So it's a general way to first compute its lower bound Hodge polygon and then determine when the Newton polygon coincides with its lower bound. 

Let $\Delta$ be an $n$-dimensional integral polytope containing the origin in $\R^n$. Define $C(\Delta)$ to be the cone generated by $\Delta$ in $\R^{n}$. For any point $u\in \R^{n}$, the weight function $w(u)$ is the smallest non-negative real number $c$ such that $u\in c\Delta$. Let $w(u)=\infty$ if such $c$ doesn't exist. Assume $\delta$ is a co-dimension 1 face of $\Delta$ not containing the origin. Let $D(\delta)$ be the least common multiple of the denominators of the coefficients in the implicit equation of $\delta$, normalized to have constant term 1.
We define the denominator of $\Delta$ to be the least common multiple of all such $D(\delta)$ given by:
\begin{equation*}
D=D(\Delta)= \mathrm{lcm}_{\delta}D(\delta)
\end{equation*}
where $\delta$ runs over all the co-dimension 1 faces of $\Delta$ that don't contain the origin. It's easy to check 
\begin{equation*}
w(\Z^n)\subseteq \frac{1}{D(\Delta)}\Z_{\geq0}\cup \{ + {\infty} \}.
\end{equation*} 
 For a non-negative integer $k$, let 
 \begin{equation*}
 W_{\Delta}(k)=\# \left\{ u \in \Z^n | w(u)= \frac{k}{D} \right\}
 \end{equation*} 
 be the number of lattice points in $\Z^n$ with weight $k/D$. 
\begin{Def}[Hodge number]\label{def5}
	Let $\Delta$ be an $n$-dimensional integral polytope containing the origin in $\R^n$. For a non-negative integer $k$, the $k$-th Hodge number of $\Delta$ is defined to be
\begin{align}\label{eq3}
	H_{\Delta}(k)=\sum^{n}_{i=0}(-1)^i \binom{n}{i}W_{\Delta}(k-iD).
\end{align}
\end{Def}
It's easy to check that 
\begin{equation*}
H_{\Delta}(k)=0, \quad \text{if}\quad k>nD.
\end{equation*}
Adolphson and Sperber\cite{AS1989} proved that $H_{\Delta}(k)$ coincides with the usual Hodge number in the toric hypersurface case that $D=1$. Based on the Hodge numbers, we define the Hodge polygon of a given polyhedron $\Delta\in \R^n$ as follows.
\begin{Def}[Hodge polygon]\label{def6}
	The Hodge polygon HP($\Delta$) of $\Delta$ is the lower convex polygon in $\R^2$ with vertices (0,0) and 
	\begin{equation*}
	Q_k=\left( \sum^k_{m=0}H_{\Delta}(m), \frac{1}{D}\sum^k_{m=0}m H_{\Delta}(m) \right), \quad  k=0,1,\ldots, nD,
	\end{equation*}
	where $H_{\Delta}(k)$ is the $k$-th Hodge number of $\Delta$, $k=0,1,\ldots, nD.$
	
	That is, HP($\Delta$) is a polygon starting from origin (0,0) with a slope $k/D$ side of horizontal length $H_{\Delta}(k)$ for $k=0,1,\ldots, nD$. The vertex $Q_k$ is called a break point if $H_{\Delta}(k+1)\neq 0$ where $k=1,2,\ldots,nD-1$. 
\end{Def}
Here the horizontal length $H_{\Delta}(k)$ represents the number of  lattice points of weight $k/D$ in a certain fundamental domain corresponding to a basis of the $p$-adic cohomology space used to compute the L-function. Adolphson and Sperber constructed the Hodge polygon and proved that it's a lower bound of the corresponding Newton polygon. 
\begin{Th}[Adolphson and Sperber\cite{AS1989}]
	For every prime p and non-degenerate Laurent polynomial $f$ with $\Delta(f)=\Delta \subset \R^n$, we have 
	\begin{equation*}
	\text{NP}(f) \geq \text{HP}(\Delta),
	\end{equation*}
	where NP($f$) is the $q$-adic Newton polygon of  $\LF^*(f,T)^{(-1)^{n-1}}.$
	Furthermore, the endpoints of NP($f$) and NP($\Delta$) coincide.
\end{Th}
\begin{Def}
	A Laurent polynomial $f$ is called ordinary if NP($f$) = HP($\Delta$). 
\end{Def}

In order to study the ordinary property, we will apply Wan's decomposition theorems\cite{Dwan1993}. Apparently, the ordinary property of a Laurent polynomial depends on its Newton polyhedron $\Delta$. Wan's theorems decompose the polyhedron $\Delta$ into small pieces which are much easier to deal with \cite{Dwan1993}. We will give a brief introduction of Wan's facial decomposition theorem and boundary decomposition theorem in subsection \ref{subsec23}.

\subsection{Wan's decomposition theorems}\label{subsec23}
 \subsubsection{Facial decomposition theorem}\label{subsubsec231}
In this paper, we use facial decomposition theorem to cut the polyhedron into small simplices. For each simplex, we can apply some criteria to determine the non-degenerate and ordinary property. 
\begin{Th}[Facial decomposition theorem\cite{Dwan1993}]\label{thm9}
	Let $f$ be a non-degenerate Laurent polynomial over $\F_q$. Assume $\Delta=\Delta(f)$ is $n$-dimensional and $\delta_1,\ldots, \delta_h$ are all the co-dimension 1 faces of $\Delta$ which don't contain the origin. Let $f^{\delta_i}$ denote the restriction of $f$ to $\delta_i$. Then $f$ is ordinary if and only if $f^{\delta_i}$ is ordinary for $1\leq i\leq h$. 
\end{Th}
In order to describe the boundary decomposition, we first express the L-function in terms of the Fredholm determinant of an infinite Frobenius matrix.\\

\subsubsection{Dwork's trace formula}\label{subsubsec232}
Let $p$ be a prime and $q=p^a$ for some positive integer $a$. Let $\Q_p$ denote the field of $p$-adic numbers and $\Omega$ be the completion of $\overline{\Q}_p$. Pick a fixed primitive $p$-th root of unity in $\Omega$ denoted by $\zeta_p$. In $\Q_p(\zeta_p)$, choose a fixed element $\pi$ satisfying
\begin{equation*}
\sum^{\infty}_{m=0}\frac{\pi^{p^m}}{p^m}=0 \quad \text{and} \quad \ord_p \pi =\frac{1}{p-1}.
\end{equation*} 
By Krasner's lemma, it's easy to check $Q_p(\pi)=Q_p(\zeta_p)$. Let $K$ be the unramified extension of $\Q_p$ of degree $a$. Let $\Omega_a$ be the compositum of $Q_p(\zeta_p)$ and $K$. 
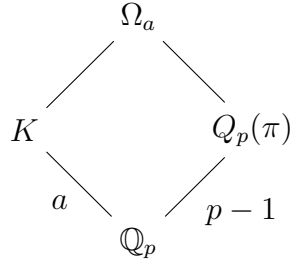
\begin{center}
	\begin{tikzpicture}[node distance = 1.5cm, auto]
      \node (Q) {$\Q_p$};
      \node (K) [above of=Q, left of=Q] {$K$};
      \node (Z) [above of=Q, right of=Q] {$Q_p(\pi)$};
      \node (O) [above of=Q, node distance = 3cm] {$\Omega_a$};
      \draw[-] (Q) to node {$a$} (K);
      \draw[-] (Q) to node [swap] {$p-1$} (Z);
      \draw[-] (K) to node {$$} (O);
      \draw[-] (Z) to node [swap] {$$} (O);
      \end{tikzpicture}
  \end{center}
Define the Frobenius automorphism $\tau \in \text{Gal}(\Omega_a/\Q_p(\pi))$ by lifting the Frobenius automorphism $x\mapsto x^p$ of Gal($\F_q/\F_p$) to a generator $\tau$ of Gal($K/\Q_p$) and extending it to $\Omega_a$ with $\tau(\pi)=\pi$. For the primitive $(q-1)$-th root of unity $\zeta_{q-1}$ in $\Omega_a$, we have $\tau(\zeta_{q-1})=\zeta_{q-1}^p$.

Let $E_p(t)$ be the Artin-Hasse exponential series,
\begin{equation*}
E_p(t)=\exp\left(\sum^{\infty}_{m=0}\frac{t^{p^m}}{p^m}\right)=\sum_{m=0}^{\infty}\lambda_m t^m \in  \Z_p[[x]].
\end{equation*}
In Dwork's terminology, a splitting function $\theta(t)$ is defined to be
\begin{equation*}
\theta(t)=E_p(\pi t)=\sum_{m=0}^{\infty}\lambda_m\pi^mt^m.
\end{equation*}
When $t=1$, $\theta(1)$ can be identified with $\zeta_p$ in $\Omega$.

Consider a Laurent polynomial $f \in $ $\F_q[x_1^{\pm1}, \ldots, x_n^{\pm1}]$ given by
\begin{equation*}
f=\sum_{j=1}^J \bar{a}_j x^{V_j},
\end{equation*}
where $V_j \in {\Z}^n$ and $\bar{a}_j \in \F_q^{*}$. Let $a_j$ be the Teichm\"{u}ller lifting of $\bar{a}_j $ in $\Omega$ satisfying $a_j^q=a_j$. Let
\begin{equation*}
F(f,x)=\prod_{j=1}^J\theta(a_j x^{V_j})=\sum_{r \in {\Z}^n} F_r(f)x^r  \in \Omega_a[[x]].
\end{equation*}
The coefficients of $F(f,x)$ are given by 
\begin{equation*}
F_r(f)=\sum_u (\prod^{J}_{j=1} \lambda_{u_j} a_j^{u_j}) \pi^{u_1+\dots+u_{J}}, \quad r \in {\Z}^n,
\end{equation*}
 where the sum is over all the solutions of the following linear system
\begin{equation*}
\sum^{J}_{j=1}u_jV_j=r \quad \text{with}\quad u_j \in \Z_{\geq 0},
\end{equation*} 
and $\lambda_m$ is $m$-th coefficient of the Artin-Hasse exponential series $E_p(t)$.

Assume $\Delta=\Delta(f)$. Let $L(\Delta)=\Z^{n}\cap C(\Delta)$ be the set of lattice points in the closed cone generated by origin and $\Delta$. For a given point $r\in \R^n$, the weight function is given by
\begin{equation*}
w(r): =\inf_{\vec{u}}\left\{ \sum_{j=1}^J u_j |\sum_{j=1}^J u_jV_j=r,\quad u_j\in \R_{\geq 0}\right\}.
\end{equation*}
In Dwork's terminology, the infinite semilinear Frobenius matrix $A_1(f)$ is a matrix whose rows and columns are indexed by the lattice points in $L(\Delta)$ with respect to the weights
\begin{equation*}
	A_1(f)=(a_{r,s}(f))=(F_{ps-r}(f)\pi^{w(r)-w(s)}),
\end{equation*}
where $r, s\in L(\Delta)$. 
Based on the fact that $\ord_p F_r(f)\geq \frac{w(r)}{p-1}$, we have the following estimate
\begin{equation*}
\ord_p( a_{r,s}(f)) \geq \frac{w(ps-r)+w(r)-w(s)}{p-1}\geq w(s).
\end{equation*} 
Let $\xi$ be an element in $\Omega$ satisfying $\xi^D=\pi^{p-1}$. Then $A_1(f)$ can be written in a block form,
\begin{equation*}
A_1(f)
=\begin{pmatrix}
A_{00} &  \xi A_{01}  & \cdots\quad & {\xi}^iA_{0i}&\cdots\\
A_{10} &  \xi A_{11}  & \cdots\quad & {\xi}^iA_{1i}&\cdots\\
 \vdots & \vdots & \ddots  & \vdots  \\
 A_{i0} &  \xi A_{i1}  & \cdots\quad & {\xi}^iA_{ii}&\cdots\\
 \vdots & \vdots & \ddots  & \vdots
\end{pmatrix},
\end{equation*}
where the block $A_{ii}$ is a $p$-adic integral $W_{\Delta}(i) \times W_{\Delta}(i)$ matrix and $W_{\Delta}(i)=\#\{u \in \Z^n | w(u)= \frac{i}{D}\}$.

The infinite linear Frobenius matrix $A_a(f)$ is defined to be
\begin{equation*}
A_a(f)=A_1(f)A_1^{\tau}(f)\cdots A_1^{\tau^{a-1}}(f).
\end{equation*}
The $q$-adic Newton polygon of $\det(I-TA_a(f))$ has a natural lower bound which can be identified with the chain level version of the Hodge polygon.
\begin{Def}Let $P(\Delta)$ be the polygon in $\R^2$ with vertices $(0,0)$ and 
\begin{equation*}
P_k=\left( \sum^k_{m=0}W_{\Delta}(m), \frac{1}{D}\sum^k_{m=0}m W_{\Delta}(m) \right), \quad  k=0,1,2, \ldots
\end{equation*}
\end{Def}
\begin{Prop}[\cite{Dwan2004}]The $q$-adic Newton polygon of $\det(I-TA_a(f))$ lies above $P(\Delta).$
\end{Prop}

By Dwork's trace formula\cite{Dwork1960}, we can identify the associated L-function with a product of some powers of the Fredholm determinant\cite{Dwan2004},
\begin{equation}\label{eq24}
\LF^{*}(f,T)^{(-1)^{n-1}}=\prod_{i=0}^{n}\det(I-Tq^{i}A_a(f))^{(-1)^i\binom{n}{i}}.
\end{equation}
Equivalently, we have,
\begin{equation}\label{eq25}
\det(I-TA_a(f))=\prod_{i=0}^{\infty} \left( \LF^{*}(f,q^iT)^{(-1)^{n-1}} \right)^{\binom{n+i-1}{i}}.
\end{equation}
\begin{Prop}[\cite{Dwan2004}]\label{prop c}
Notations as above. Assume $f$ is non-degenerate with $\Delta=\Delta(f)$. Then $\mathrm{NP}(f)=\mathrm{HP}(\Delta)$ if and only if the $q$-adic Newton polygon of $\det(I-TA_a(f))$ coincides with its lower bound $P(\Delta).$
\end{Prop}

\subsubsection{Boundary decomposition}\label{subsubsec233}
Let $f \in $ $\F_q[x_1^{\pm1}, \ldots, x_n^{\pm1}]$ with $\Delta=\Delta(f)$, where $\Delta$ is an $n$-dimensional integral convex polyhedron in $\R^n$ containing the origin. Let $C(\Delta)$ be the cone generated by $\Delta$ in $\R^n.$ 
\begin{Def}\label{defbd}
	The boundary decomposition 
	\begin{equation*}
	B(\Delta)=\{ \text{ the interior of a closed face in }C(\Delta) \text{ containing the origin} \}
	\end{equation*}
	is the unique interior decomposition of $C(\Delta)$ into a disjoint union of relatively open cones. 
\end{Def}
If the origin is a vertex of $\Delta$, then it is the unique 0-dimensional open cone in $B(\Delta)$. Recall that $A_1(f)=(a_{r,s}(f))$ is the infinite semilinear Frobenius matrix whose rows and columns are indexed by the lattice points in $L(\Delta)$. For $\Sigma \in B(\Delta)$, we define $A_1(\Sigma,f)$ to be the submatrix of $A_1(f)$ with $r,s \in \Sigma$. Let $f^{\overline{\Sigma}}$ be the restriction of $f$ to the closure of $\Sigma$. Then $A_1(\Sigma,f^{\overline{\Sigma}})$ denotes the submatrix of $A_1(f^{\overline{\Sigma}})$ with $r,s \in \Sigma$.

Let $B(\Delta)=\{\Sigma_0, \ldots, \Sigma_h\}$ such that $\text{dim}(\Sigma_i)\leq \text{dim}(\Sigma_{i+1})$, $i=0,\ldots,h-1.$ Define $B_{ij}=(a_{r,s}(f))$ with $ r\in \Sigma_i$ and $ s\in \Sigma_j$ $(0\leq i,j \leq h)$. After permutation, the infinite semilinear Frobenius matrix can be written as
\begin{equation}
A_1(f)=
\begin{pmatrix}
B_{00} &  B_{01}  & \cdots\quad &B_{0h}\\
B_{10} &  B_{11}  & \cdots\quad & B_{1h}\\
 \vdots & \vdots & \ddots  & \vdots  \\
B_{h0} & B_{h1}  & \cdots\quad & B_{hh}
\end{pmatrix},
\end{equation} 
where $B_{ij}=0$ for $i>j$. Then $\det(I-TA_1(f))=\prod_{i=0}^h\det(I-TB_{ii})$ and we have the boundary decomposition theorem.
\begin{Th}[Boundary decomposition\cite{Dwan1993}] 
	Let $f\in \F_q[x_1^{\pm1}, \ldots, x_n^{\pm1}]$ with $\Delta=\Delta(f)$. Then we have the following factorization
	\begin{equation*}
	\det(I-TA_1(f))=\prod_{\Sigma \in B(\Delta)}\det \left(I-TA_1(\Sigma,f^{\overline{\Sigma}})\right).
	\end{equation*}
\end{Th}
\begin{Cor}\label{coro12}
	Let $f \in \F_q[x_1^{\pm1}, \ldots, x_n^{\pm1}]$ be a non-degenerate Laurent polynomial with an $n$-dimensional Newton polyhedron $\Delta$. If the origin is a vertex of $\Delta$, then the associated L-function
	\begin{equation*}
	\LF^*(f,T)^{(-1)^{n-1}}=(1-\psi(\Tr(c))T)\prod^{n! \Vol(\Delta(f))-1}_{i=1}(1-\alpha_i T),
	\end{equation*}
	where $c$ is the constant term of $f$, $\Tr: \F_{q} \rightarrow \F_{p}$ is the trace map and $|\alpha_i| \leq q^{n/2}$.
\end{Cor}
\begin{proof}
By Formula (\ref{eq24}),
\begin{equation*}
\LF^{*}(f,T)^{(-1)^{n-1}}=\prod_{i=0}^{n}\det(I-Tq^{i}A_a(f))^{(-1)^i\binom{n}{i}},
\end{equation*}
where $A_a(f)=A_1(f)A_1^{\tau}(f)\cdots A_1^{\tau^{a-1}}(f).$

Consider the boundary decomposition $B(\Delta)$. Let $B(\Delta)=\{\Sigma_0, \ldots, \Sigma_h\}$ such that $\text{dim}(\Sigma_i)\leq \text{dim}(\Sigma_{i+1})$. If the origin is a vertex of $\Delta$, then $\Sigma_0$ is the origin and $\Sigma_0=\overline{\Sigma}_0$. The restriction polynomial $f^{\overline{\Sigma}_0}=c$ where $c$ is the constant term of $f$. By boundary decomposition theorem, it's easy to get
\begin{align*}
\det(I-TA_a(f))
&=\prod_{\Sigma \in B(\Delta)}\det \left(I-TA_a(\Sigma,f^{\overline{\Sigma}})\right)\\
&=\det \left(I-TA_a(\Sigma_0,f^{\overline{\Sigma}_0})\right)  
\prod_{\substack{{\Sigma \in B(\Delta)}\\ \Sigma \neq \Sigma_0}} \det \left(I-TA_a(\Sigma,f^{\overline{\Sigma}})\right).
\end{align*}
Since the $k$-th exponential sum
\begin{equation*}
S_k^*\left(f^{\overline{\Sigma}_0}\right)=\psi\left(\Tr_k\left(f^{\overline{\Sigma}_0}\right)\right)=\psi\left(\Tr(c)\right)^k
\end{equation*}
where $\Tr_k$ is the trace map from $\F_{q^k}$ to $\F_{p}$, the corresponding L-function is given by 
\begin{equation*}
\LF^*(f^{\overline{\Sigma}_0},T)=\frac{1}{1-\psi(\Tr(c))T}.
\end{equation*}
From Formula (\ref{eq24}),
\begin{equation*}
\det\left(I-TA_a(\Sigma_0,f^{\overline{\Sigma}_0})\right)
=\LF^*(f^{\overline{\Sigma}_0},T)^{-1}=1-\psi(\Tr(c))T.
\end{equation*}
Combining with Theorem \ref{thm2}, we have
\begin{equation*}
\LF^*(f,T)^{(-1)^{n-1}}=(1-\psi(\Tr(c))T)\prod^{n! \Vol(\Delta(f))-1}_{i=1}(1-\alpha_i T),
\end{equation*}
where $\alpha_i \in \Z[\zeta_p]$ and $|\alpha_i| \leq q^{n/2}$. The corollary then follows.
\end{proof}
\medskip

\subsection{Diagonal local theory}\label{subsec24}
In this subsection, we give some non-degenerate and ordinary criteria for the diagonal Laurent polynomials whose Newton polyhedrons are simplices.
\begin{Def}
	A Laurent polynomial $f \in  \F_{q}[x_1^{\pm1},\ldots, x_n^{\pm1}]$ is called diagonal if $f$ has exactly $n$ non-constant terms and $\Delta(f)$ is an $n$-dimensional simplex in $\R^n.$
\end{Def}
Let $f$ be a Laurent polynomial over  $\F_{q}$,
\begin{equation*}
f(x_1,x_2, \ldots x_n)=\sum_{j=1}^n a_j x^{V_j},
\end{equation*}
 where $a_j \in \F ^*_{q}$ and $V_j=(v_{1j},\ldots, v_{nj})\in \Z^n,\quad  j=1,2,\ldots,n.$ Let $\Delta=\Delta(f)$. Then the vertex matrix of $\Delta$ is defined to be
\begin{equation*}
M(\Delta)=(V_1,\ldots,V_n),
\end{equation*} 
where the $i$-th column is the $i$-th exponent of $f$. If $f$ is diagonal, $M(\Delta)$ is invertible.
 \begin{Prop}[\cite{Dwan2004}]\label{prop c1}
	Suppose $f\in \F_{q}[x_1^{\pm1},\ldots, x_n^{\pm1}]$ is diagonal with $\Delta=\Delta(f)$. Then $f$ is non-degenerate if and only if $p$ is relatively prime to $\det(M(\Delta))$. 
\end{Prop}
Let $S(\Delta)$ be the solution set of the following linear system
\begin{align*}
M(f)
\begin{pmatrix}
r_1\\ r_2 \\ \vdots \\ r_n
\end{pmatrix}
\equiv 0 (\bmod1),\quad r_i \in \Q \cap [0,1).
\end{align*}
It's easy to prove that $S(\Delta)$ is an abelian group and its order is given by
\begin{align}\label{eq2}
	\left|\det{M(f)}\right|=n!\Vol(\Delta).
\end{align}
By the fundamental structure theorem of finite abelian group, we decompose $S(\Delta)$ into a direct product of invariant factors,
\begin{equation*}
S(\Delta)=\bigoplus_{i=0}^n\Z/d_i\Z,
\end{equation*}
where $d_i|d_{i+1}$ for $i=1,2,\ldots, n-1.$ By the Stickelberger theorem for Gauss sums, we have the following ordinary criterion for a non-degenerate Laurent polynomial\cite{Dwan2004}.
\begin{Prop}[\cite{Dwan2004}]\label{prop15}
	Suppose $f\in \F_{q}[x_1^{\pm1},\ldots, x_n^{\pm1}]$ is a non-degenerate diagonal Laurent polynomial with $\Delta=\Delta(f)$. Let $d_n$ be the largest invariant factor of $S(\Delta)$. If $p\equiv 1 (\bmod d_n)$, then $f$ is ordinary at $p$. 
\end{Prop}

\section{Proof of the Main Theorem}\label{sec3}
We prove the main theorem in this section. Recall that the exponential sum $S_k(\vec{a})$ has the expression,
\begin{equation*}
S_k(\vec{a})=\sum_{\substack{{\frac{a_5}{x_1x_2}+\frac{a_6}{x_3x_4}=1}\\x_i \in \F^*_{q^k}}}\psi\left(\Tr_k\left(a_1x_1+a_2x_2+a_3x_3+a_4x_4\right)\right),
\end{equation*}
where $\psi: \F_{p} \rightarrow \mathbb{C}^*$ is a nontrivial additive character, $\Tr_k: \F_{q^k} \rightarrow \F_{p}$ is the trace map and $a_j\in \F_{q}^*$.

Our main purpose is to determine the weights and slopes of the reciprocal roots of the generating L-function corresponding to $S_k(\vec{a})$ using Wan's decomposition theorems. Let $g\in \F_{q^k}[x_1^{\pm1},\ldots, x_5^{\pm1}]$ be the Laurent polynomial defined by  
\begin{equation*}
g=a_1x_1+a_2x_2+a_3x_3+a_4x_4+x_5\left(\frac{a_5}{x_1x_2}+\frac{a_6}{x_3x_4}-1\right).
\end{equation*}
The associated exponential sum of $g$ is 
\begin{equation*}
S_k^*(g)=\sum_{x_i\in \F^*_{q^k}}\psi\left(\Tr_k\left(a_1x_1+a_2x_2+a_3x_3+a_4x_4+x_5\left(\frac{a_5}{x_1x_2}+\frac{a_6}{x_3x_4}-1\right)\right)\right).
\end{equation*}
Let $\LF(\vec{a},T)$ be the generating L-function of $S_k(\vec{a})$ and $\LF^*(g,T)$ be the generating L-function of $S_k^*(g)$. Since 
\begin{align*}
&S_k(\vec{a})\\
=&\frac{1}{q^k}\sum_{\substack{{x_1, x_2, x_3, x_4\in \F^*_{q^k}}\\x_5 \in\F_{q^k}}}\psi\left(\Tr_k\left(a_1x_1+a_2x_2+a_3x_3+a_4x_4+x_5\left(\frac{a_5}{x_1x_2}+\frac{a_6}{x_3x_4}-1\right)\right)\right)\\
=&\frac{1}{q^k}+\frac{1}{q^k}\sum_{x_i\in \F^*_{q^k}}\psi\left(\Tr_k\left(a_1x_1+a_2x_2+a_3x_3+a_4x_4+x_5\left(\frac{a_5}{x_1x_2}+\frac{a_6}{x_3x_4}-1\right)\right)\right),
\end{align*}
we have the following relationship between $S_k(\vec{a})$ and $S_k^*(g)$,
\begin{align}\label{eq21}
S_k(\vec{a})=\frac{1}{q^k}+\frac{1}{q^k}S_k^*(g),
\end{align}
which implies
\begin{align*}
	\LF(\vec{a},T)=\frac{1}{1-T/q}\LF^*\left(g,T/q\right).
\end{align*}
So to estimate the reciprocal roots of $\LF(\vec{a},T)$, it's sufficient to evaluate $\LF^*(g,T).$ In order to obtain the weights and slopes of the reciprocal roots of $\LF(\vec{a},T)$, we consider the Newton polygon of $\LF(\vec{a},T)$.

\subsection{Slopes of $\LF(\vec{a},T)$}\label{subsec31}
Hereinafter, let $\Delta=\Delta(g)$ denote the Newton polyhedron corresponding to the Laurent polynomial $g$. Claim that $\dim\Delta=5$ and $\Delta$ has 8 vertices. Obviously, the non-zero vertices of $\Delta$ are $V_1=(1,0,0,0,0)$, $V_2=(0,1,0,0,0)$, $V_3=(0,0,1,0,0)$, $V_4=(0,0,0,1,0)$, $V_5=(0,0,0,0,1)$, $V_6=(-1,-1,0,0,1)$ and $V_7=(0,0,-1,-1,1)$. Since the origin can not be expressed as a linear combination of $V_j (1\leq j\leq 7)$ with non-negative coefficients, it is also a vertex of $\Delta$. 

We now turn to the faces of $\Delta$. Let $\delta$ be a face of $\Delta$ not containing the origin. Note that $\dim \delta=$ rank of the vertex matrix $M(\delta)$. As a face of $\Delta$, $\delta$ satisfies the criterion in Proposition \ref{prop20}.
\begin{Prop}\label{prop20}
	Let $\Delta$ be an $n$-dimensional convex polyhedron in $\R^n$ and $\delta$ be a co-dimension $1$ face of $\Delta$ not containing the origin. Let $h(x)=\sum^{n}_{i=1}e_ix_i=1$ be the equation of $\delta$, where $e_i$ are uniquely determined rational numbers not all zero.  For any vertex $V$ of $\Delta$, we have $h(V)\leq1$.
\end{Prop} 
Restricted by Proposition \ref{prop20}, $\Delta$ has 9 codimension 1 faces not containing the origin, denoted by $\delta_i (1\leq i\leq 9).$ The equations of $\delta_i$ (listed as follows) and the corresponding vertices (see appendix) can be directly obtained using computer programming.
\allowdisplaybreaks
\begin{align*}
\delta_1:\quad &x_1+x_2+x_3+x_4+x_5=1,\\
\delta_2:\quad &x_1+x_2+x_3-x_4+x_5=1,\\
\delta_3:\quad &x_1+x_2-x_3+x_4+x_5=1,\\
\delta_4:\quad &x_1-x_2+x_3+x_4+x_5=1,\\
\delta_5:\quad &x_1-x_2+x_3-x_4+x_5=1,\\
\delta_6:\quad &x_1-x_2-x_3+x_4+x_5=1,\\
\delta_7:\quad &-x_1+x_2+x_3+x_4+x_5=1,\\
\delta_8:\quad &-x_1+x_2+x_3-x_4+x_5=1,\\
\delta_9:\quad &-x_1+x_2-x_3+x_4+x_5=1.
\end{align*}
Recall that the restriction of $g$ to $\delta_i$ is defined by
\begin{equation*}
g^{\delta_i}=\sum_{V_j\in \delta_i}a_j x^{V_j}.
\end{equation*}
It is explicit from the equations of $\delta_i$ that the denominator $D=D(\Delta)=1$. For $1\leq i\leq9$, note that each 
$\left| \det M(\delta_i)\right|=1$ and Laurent polynomial $g^{\delta_i}$ is diagonal.
According to Proposition \ref{prop c1} and \ref{prop15}, $g^{\delta_i}$ is non-degenerate and ordinary. From Theorem \ref{thm9}, $g$ is ordinary which means that the Newton polygon of $\LF^*(g,T)$ coincides with its Hodge polygon. In order to determine the slopes of the reciprocal roots of $\LF^*(g,T)$, we compute the Hodge numbers of $\Delta$.
\begin{Th}\label{thm21}
	Notations as above. Let $k$ be any non-negative integer.
	\begin{enumerate}[(i)]
		\item The Volume of $\Delta$ is
		$\Vol(\Delta)=\frac{9}{5!}.$
		\item Let $W_{\Delta}(k)$ be the number of lattice points in $\Z^n$ with weight $\frac{k}{D}$. For $k\leq5$, we have $W_{\Delta}(0)=1$, $W_{\Delta}(1)=7$, $W_{\Delta}(2)=28$, $W_{\Delta}(3)=82$, $W_{\Delta}(4)=196$ and $W_{\Delta}(5)=406$.
		\item The Hodge numbers of $\Delta$ are $H_{\Delta}(0)=1$, $H_{\Delta}(1)=2$, $H_{\Delta}(2)=3$, $H_{\Delta}(3)=2$, $H_{\Delta}(4)=1$ and $H_{\Delta}(k)=0$ for $k\geq5$.
	\end{enumerate}
\end{Th}
\begin{proof}
	For any integer $i$ ($1\leq i\leq9$), let $\Delta_i$ be the polytope generated by $\delta_i$ and the origin. The facial decomposition of $\Delta$ \cite{Dwan1993} is defined by
	\begin{align*}
		\Delta=\bigcup^{9}_{i=1}\Delta_i.
	\end{align*}	
	It follows that the volume of $\Delta_i$ satisfies the relationship
	\begin{align*}
		\Vol(\Delta)=\sum_{i=1}^{9}\Vol(\Delta_i).
	\end{align*}
	Formula (\ref{eq2}) implies $\Vol(\Delta_i)=\frac{1}{5!}$ for all $\Delta_i(1\leq i\leq9)$. Then we have $\Vol(\Delta)=\frac{9}{5!}$.	
	
	Next we calculate $W_{\Delta}(k)$. For any lattice points $u\in L(\Delta) = \Z^5 \cap C(\Delta)$, let $w_{\Delta}(u)$ be the weight function of $u$ with respect to $\Delta$. Recall that 
	\begin{equation*}
	w_\Delta(u)=\inf_{\vec{c}}\left\{ \sum_{j=1}^7 c_j \bigg\rvert \sum_{j=1}^7 c_jV_j=u,\quad c_j\in \R_{\geq 0}\right\},
	\end{equation*}
	where $V_j$ $(1\leq j \leq 7)$ denote all the non-zero vertices of $\Delta$. Since the denominator $D=1$, we have
	\begin{align*}
		W_{\Delta}(k)=\#\left\{ u\in L(\Delta) \rvert w_{\Delta}(u)=k\right\}.
	\end{align*}
	For simplicity, we decompose $\Delta$ into $\Delta_i$ $(1\leq i\leq 9)$,i.e.,
	\begin{align*}
		\left\{ u\in L(\Delta) \rvert w_{\Delta}(u)=k\right\} =\bigcup^9_{i=1}\left\{u_i\in L(\Delta_i)\ \rvert\  w_{\Delta_i}(u_i)=k\right\}.
	\end{align*}
	Any lattice point $u_i \in L(\Delta_i)$ satisfying $w_{\Delta_i}(u_i)=k$ can be expressed in the following form
	\begin{align}\label{eq30}
	u_i=\sum_{j=1}^{5}c_jV_{ij},
	\end{align}
	where $\sum_{j=1}^{5}c_j=k$ for $c_j\in \Z_{\geq 0}$ and $V_{ij}$ denote the non-zero vertices of $\delta_i$ (see appendix). For $k\leq5$, we exhaust all the lattice points $u$ satisfying formula (\ref{eq30}) and take union of them. Through computer calculation, we obtain $W_{\Delta}(0)=1$, $W_{\Delta}(1)=7$, $W_{\Delta}(2)=28$, $W_{\Delta}(3)=82$, $W_{\Delta}(4)=196$ and $W_{\Delta}(5)=406$. 
	
	The Hodge number $H_{\Delta}(k)$ follows from $W_{\Delta}(k)$ and formula (\ref{eq3}).
\end{proof}

Theorem \ref{thm2} and the first result of Theorem \ref{thm21} show that the degree of $\LF^*(g,T)$ is 9. For $1\leq i \leq 8$, let $\alpha_i$ denote the reciprocal roots of $\LF(\vec{a},T)$ and let $\beta_i $ denote the nontrivial reciprocal roots of $\LF^*(g,T)$. Now we factor $\LF(\vec{a},T)$ and $\LF^*(g,T)$ as follows.
\begin{Prop}\label{prop21}
	Notations as above. We have
	\begin{align*}
		\LF(\vec{a},T)&=\prod^{8}_{i=1}\left(1-\alpha_iT\right),\\
		\LF^*(g,T)&=(1-T)\prod^{8}_{i=1}\left(1-\beta_iT\right),
	\end{align*}
	where $\left|\beta_i\right|\leq q^{\frac{5}{2}}$ and $ \alpha_i=\frac{\beta_i}{q}$. 
\end{Prop}
\begin{proof}
	From Theorem \ref{thm2}, one has
	\begin{align*}
		\LF^*(g,T)=\exp\left(\sum^{\infty}_{k=1}S_k^*(g)\frac{T^k}{k}\right)=\prod^{5!\Vol(\Delta(g))}_{i=1}\left(1-\beta_iT\right)
		=(1-T)\prod^{8}_{i=1}\left(1-\beta_iT\right),
	\end{align*}
	where the last equality can be deduced from Corollary \ref{coro12} that $\LF^*(g,T)$ has a trivial unit root $\beta_0=1$.
	Substituting formula (\ref{eq21}) into $\LF(\vec{a},T)$, we obtain
	\begin{align*}
	\LF(\vec{a},T)&=\exp\left(\sum^{\infty}_{k=1}S_k(\vec{a})\frac{T^k}{k}\right)
	=\exp\left(\sum^{\infty}_{k=1}\frac{T^k}{kq^k}\right)\exp\left(\sum^{\infty}_{k=1}S_k^*(g)\frac{T^k}{kq^k}\right)\\
	&=\frac{1}{1-\frac{T}{q}}\LF^*\left(g,\frac{T}{q}\right)
	=\prod^{8}_{i=1}\left(1-\frac{\beta_iT}{q}\right).
	\end{align*} 
	\qedhere
\end{proof}

The next corollary gives the slopes of $\alpha_i$.
\begin{Cor}\label{cor22}
	Notations as above. Let $h_k$ be the number of $\alpha_i$ with slope $k$, i.e. 
	\begin{equation*}
	h_k=\#\{\alpha_i\ \rvert\ \ord_q\alpha_i=k\}.
	\end{equation*} 
	We have $h_0=2$, $h_1=3$, $h_2=2$, $h_3=1$ and $h_k=0$ for $k\geq4$.
\end{Cor}
\begin{proof}
	Since $g$ is ordinary, the Newton polygon of $\Delta$ coincides with its Hodge polygon. In this case, the slope $k$ segment in Newton polygon has horizontal length $H_{\Delta}(k)$. By Lemma
	\ref{le4}, $H_{\Delta}(k)$ equals the number of $\beta_i$ such that  $\ord_q\beta_i=k$ . Values of $h_k$ follow from $H_{\Delta}(k)$
	since $\ord_q(\alpha_i)=\ord_q(\beta_i)-1$.
\end{proof}

\subsection{Weights of $\LF(\vec{a},T)$}\label{subsec32}
In this subsection, we refer 3 methods to determine weights of $\beta_i$ and demonstrate the first one in detail. Weights of $\alpha_i$ can be deduced directly from $\beta_i$.

For a Laurent polynomial $f$, the generating L-function $\LF^*(f,T)$ can be expressed as a product of some powers of the Fredholm determinant using Dwork's trace formula. See subsection \ref{subsec23} for details. With the help of Wan's boundary decomposition theorem, we identify that $\LF^*(g,T)$ has three real reciprocal roots, $1$, $q$, $q^2$, using formula (\ref{eq24}). Restricted by the Hodge numbers, weights of $\beta_i$ can be obtained.
\begin{Lemma}\label{le34}
	Let $f$ be a non-degenerate odd Laurent polynomial with an $n$-dimensional Newton polyhedron $\Delta$. The associated L-function is of the following form,
	\begin{equation*}
	\LF^*(f,T)^{(-1)^{n-1}}=\prod^{n! \Vol(\Delta(f))}_{i=1}(1-\alpha_i T),
	\end{equation*}
	where $|\alpha_i|=q^{\omega_i/2}$ and $\omega_i\in \Z\cap [0,n]$. 
	If $\alpha_i$ is real, then
	\begin{align}\label{eqr}
	\ord_q\alpha_i=\frac{\omega_i}{2} \leq \frac{n}{2}.
	\end{align}
	If $\alpha_i$ is non-real, then the conjugate $\overline{\alpha}_i$ is also a reciprocal root of $\LF^*(f,T)^{(-1)^{n-1}}$ and 
	\begin{align}\label{eqc}
	\ord_q\alpha_i+\ord_q\overline{\alpha}_i=\omega_i \leq n.
	\end{align}
\end{Lemma}
\begin{proof}
By the assumption, we know $f(-x)=-f(x)$ which implies 
\begin{align*}
\LF^*(f,T)=\overline{\LF^*(f,T)}.
\end{align*} The lemma then follows.
\end{proof}
\begin{Th}\label{thm35}
	Notations as above. We have
	\begin{align*}
	\LF^*(g,T)=(1-T)(1-qT)(1-q^2T)\prod^6_{i=1}(1-\beta_iT).
	\end{align*}
	where $|\beta_i| = q^{5/2}$, $i=1,2,\ldots, 6$.
\end{Th}
\begin{proof}
	Denote $D(q^iT)=\det\left(I-q^iTA_a(g)\right)$. First we compute $D(T)$ using boundary decomposition theorem.
	Recall the boundary decomposition $B(\Delta)$ defined in definition \ref{defbd}. Let $N(i)$ denote the number of $i$-dimensional face $\Sigma_{i,j}$ of $C(\Delta)$, where $0\leq i\leq \dim\Delta$ and $1\leq j\leq N(i)$. For Newton polyhedron $\Delta=\Delta(g)$ in our example, $N(i)$ equals the number of $i$-dimensional faces of $\Delta$ containing the origin, which yields $N(0)=1$, $N(1)=6$, $N(2)=15$ (see Table \ref{tab5} in appendix). The following proof with respect to face $\Sigma_{i,j}$ is  independent of the choice of $j$. For simplicity, we abbreviate $\Sigma_{i,j}$ to $\Sigma_i$. Note that $\Sigma_i$ is an open cone and $\Sigma_i\in B(\Delta)$. Let $\overline{\Sigma}_i$ be the closure of $\Sigma_i$.
	Denote 
	\begin{equation*}
	D_i'(T)=\det\left(I-TA_a(\Sigma_i,g^{\overline{\Sigma}_i})\right)\quad 
	\text{and}\quad 
	D_i(T)=\det\left(I-TA_a(\overline{\Sigma}_i,g^{\overline{\Sigma}_i})\right).
	\end{equation*}	
	It is obvious that the unique $0$-dimensional cone $\overline{\Sigma}_0$ is the origin
	and $D_0(T)=D_0'(T)=1-T$. When $i=1$, each $g^{\overline{\Sigma}_1}$ can be normalized to $x$ by variable substitution. Explicitly,
	\begin{align*}
	\LF^*(g^{\overline{\Sigma}_1},T)
	=\exp \left(\sum^\infty _ {k=1}-\frac{T^k}{k}\right)=1-T.
	\end{align*}
	By formula (\ref{eq25}), we have
	\begin{align*}
	D_1(T) &= \prod_{i=0}^{\infty} \left( \LF^{*}\left(g^{\overline{\Sigma}_1},q^iT\right)\right)^{\binom{i}{i}}
	= (1-T)\left(1-qT\right)\left(1-q^2T\right)\left(1-q^3T\right)\cdots.
	\end{align*} 
	Since the only boundary of $\overline{\Sigma}_1$ is $\overline{\Sigma}_0$, we get $D_1'(T)$ after eliminating $D_0'(T)$, i.e., 
	\begin{align*}
	D_1'(T)=\frac{D_1(T)}{D_0'(T)}=\left(1-qT\right)\left(1-q^2T\right)\prod_{i=3}^{\infty}\left(1-q^iT\right),
	\end{align*}	
	Similarly, we transfer each $g^{\overline{\Sigma}_2}$ to $x_1+x_2$. Then  $\LF^*(g^{\overline{\Sigma}_2},T)=\frac{1}{1-T}$ and 
	\begin{align*}
	D_2(T) &= \prod_{i=0}^{\infty} \left( \LF^{*}\left(g^{\overline{\Sigma}_2},q^iT\right)\right)^{-\binom{1+i}{i}}
	= (1-T)\left(1-qT\right)^2\left(1-q^2T\right)^3\left(1-q^3T\right)^4\cdots.
	\end{align*} 
	Boundary of $\overline{\Sigma}_2$ consists of the origin and 2 sides, i.e. $\overline{\Sigma}_1$, which yields
	\begin{align*}
	D_2'(T)=\frac{D_2(T)}{D_0'(T)D_1'^2(T)}=\left(1-q^2T\right)\prod_{i=3}^{\infty}\left(1-q^iT\right)^{i-1}.
	\end{align*}
	For $3\leq i\leq 5$, $\overline{\Sigma}_i$ is not necessarily a simplex cone. By Proposition \ref{prop c} and Theorem \ref{thm21}, 
	\begin{align}
	D(T)=\prod^{\infty}_{i=1}\left(1-\gamma_iT\right),
	\end{align}
	where $\#\{\gamma_i\ \rvert\   \ord_q\gamma_i=k\}=W_{\Delta}(k)$ for $k\in \Z_{\geq0}$.
	Using Wan's boundary decomposition theorem,
	\begin{align*}
	D(T)&=\prod^{5}_{i=1}\prod^{N(i)}_{j=1}D_i'(T)=D_0'(T)D_1'(T)^6D_2'(T)^{15}\cdots\\
	&=(1-T)(1-qT)^6(1-q^2T)^{21}(1-\gamma_1T)h_1(T),
	\end{align*}
	where $\left|\gamma_1\right|_q=q^{-1}$ and the slopes of reciprocal roots of $h_1(T)$ are greater than 1. From formula (\ref{eq24}), we obtain
	\begin{align*}
	\LF^*(g,T)&=\frac{D(T)D(q^2T)^{10}D(q^4T)^{5}}{D(qT)^5D(q^3T)^{10}D(q^5T)}\\
	&=\frac{(1-T)(1-qT)^6(1-q^2T)^{21}(1-\gamma_1T)(1-q^2T)^{10}\cdots}{(1-qT)^5(1-q^2T)^{30}(1-\gamma_1qT)^5\cdots}\\
	&=(1-T)(1-qT)\frac{(1-q^2T)(1-\gamma_1T)}{(1-\gamma_1qT)^5}h_2(T).
	\end{align*}
	where the slopes of reciprocal roots of $h_2(T)$ are greater than 1.
	
	We claim that $\gamma_1\neq q$ and prove it by contradiction. Suppose $\gamma_1=q$, then $\LF^*(g,T)$ has no complex reciprocal roots with slopes $\leq1$. Let $\beta_6$ be the unique slope 4 reciprocal root of $\LF^*(g,T)$. By formula (\ref{eqr}) and (\ref{eqc}), $\beta_6$ is non-real with weight $\geq 6$ but $\leq 5$ which leads to a contradiction. Combining with Theorem \ref{thm2}, we get
	\begin{align}\label{eq36}
	\LF^*(g,T)=(1-T)(1-qT)(1-q^2T)\prod^6_{i=1}(1-\beta_iT).
	\end{align}

	It can be deduced from Hodge number $H_\Delta(k)$ and Lemma \ref{le34} that $\beta_i$ is non-real for $1\leq i \leq 6$. Let $\omega_i'$ be the weight of $\beta_i$.
	Notice that
	$\prod_{i=1}^6\beta_i=q^{\frac{\sum_{i=1}^6\omega_i'}{2}}$ and $0\leq \omega_i' \leq 5$. 
	By $H_\Delta(k)$ and formula (\ref{eq36}), we have
	\begin{equation*}
	\sum_{i=1}^6\omega_i'=2 \ord_q(\prod_{i=1}^6\beta_i)=2\sum_{i=1}^6\ord_q\beta_i=30,
	\end{equation*}
	which implies $\omega_i'=5$ for $1\leq i\leq 6$. 	
\end{proof}
Then our main theorem follows from Proposition \ref{prop21} and Corollary \ref{cor22}.
\begin{Th}\label{thm5}
Notations as above. We have
         \begin{align*}
	\LF(\vec{a},T)=(1-T)(1-qT)\prod^6_{i=1}(1-\alpha_iT),
	\end{align*}
	where $|\alpha_i| = q^{3/2}$ 
	for $1\leq i \leq 6$. Weights and slopes of $\alpha_i$ are given by Table \ref{tab1}.
	\begin{table}[htbp]	
		\centering
		\allowdisplaybreaks
		\begin{tabular}{cccccc}
			\toprule 
			$\alpha_i$ & $\{\alpha_1, \alpha_6\}$ & $\{\alpha_2, \alpha_4\}$ & $\{\alpha_3, \alpha_5\}$  \\ 
			\midrule 
			\text{Slopes} & $\{0, 3\}$ & $\{1, 2\}$ & $\{1, 2\}$  \\ 
			\text{Weights} & $3$ & $3$ & $3$ \\ 
			$|\alpha_i|_p$ & $\{1, q^{-3}\}$ & $\{q^{-1}, q^{-2}\}$ & $\{q^{-1}, q^{-2}\}$  \\ 
			$|\alpha_i|$ &  $\{q^{\frac{3}{2}},q^{\frac{3}{2}}\}$ & $\{q^{\frac{3}{2}},q^{\frac{3}{2}}\}$ & $\{q^{\frac{3}{2}},q^{\frac{3}{2}}\}$ \\ 
			\bottomrule
		\end{tabular}
		\caption{Slopes and Weights of $\alpha_i$}\label{tab1}
	\end{table}\\
	in which the reciprocal roots $\alpha_i$ are enumerated with respect to their $q$-adic slopes.
\end{Th}

Let $e_\omega(0\leq \omega \leq5)$ be the number of $\beta_i$ with weight $\omega$. The second way is to calculate $e_\omega$ by Denef-Loeser's weight formula obtained using intersection cohomology \cite{Denef1991}. This formula works for non-degenerate Laurent polynomial $f$ with respect to the Newton polyhedron $\Delta(f)$. When $\dim\left(\Delta(f)\right)\geq5$ and the $\Delta(f)$ is not simple at the origin, this formula becomes combinatorially complicated to use, which is precisely the case for our example $g$.
After cumbersome computation, we obtain $e_0=1$, $e_1=0$, $e_2=1$, $e_3=0$, $e_4=1$, $e_5=6$. This method also leads to weights of all $\beta_i$, while it cannot determine the exact values of all real roots.

Katz's 
calculation of the local monodromy of the Kloosterman sheaf provides the third strategy \cite{katz}. The rank $3$ Kloosterman sum is defined by
\begin{align*}
Kl(a,b,c) = \sum_{x_1,x_2\in \F^*_{q^k}}\psi\left(\Tr_k\left(ax_1+bx_2+\frac{c}{x_1x_2}\right)\right),
\end{align*}
where $a$,$b$,$c\in\F^*_{q}$, $\psi$ and $\Tr_k$ are as defined in this paper.
If the variable $x_5$ in $S^*_k(g)$ is fixed, the exponential sum symmetrically splits into a product of two rank $3$ Kloosterman sums, i.e.
\begin{align*}
S^*_k (g) = \sum_{x_5\in \F^*_{q^k}}Kl(a_1,a_2,a_5x_5)Kl(a_3,a_4,a_6x_5)\psi\left(\Tr_k\left(-x_5\right)\right).
\end{align*}
One then needs to calculate the weights of the cohomology on $G_m$ of 
the tensor product of two Kloosterman sheaves with an Artin-Schreier 
sheaf. This reduces to understanding the local monodromy at $0$ and 
infinity of this tensor product. 
Katz's strategy arrives at the same result as Wan's and is more advanced.

As a corollary of Theorem \ref{thm5}, we propose an estimation for the exponential sums $S_k(\vec{a})$.
\begin{Th}
	Notations as above. We have 
	\begin{align*}
		\left|S_k(\vec{a})\right| \leq 6q^{\frac{3k}{2}}+q^k+1.
	\end{align*}
\end{Th}

Denef and Loeser showed in \cite{Denef1991} that the conjecture about the number of weight $k$ reciprocal roots of an $n$-dimensional polytope ($k\leq n \in \Z$) by Adolphson and Sperber \cite{Adolphson1990} is false for some 5-dimensional simplicial Newton polyhedron. They proved the existence of counterexample for $n=5$, but didn't propose an explicit one. Our example, in fact, describes the first explicit counterexample for Adolphson-Sperber's conjecture. 
\begin{Conj}[Adolphson-Sperber, \cite{Adolphson1990}]\label{conj9}
	Let $f \in \mathbf{F}_{q}\left[x_{1}, \ldots, x_{n},\left(x_{1} \ldots x_{n}\right)^{-1}\right]$ be nondegenerate and suppose $\operatorname{dim} \Delta(f)=n .$ Let $w_k$ be the number of roots of $\LF^*(f,T)$ with weight $k$, where $0\leq k\leq n$. Then
	\begin{equation*}
	w_{k}=(-1)^{k} \sum_{0 \in \sigma\ \text{face of}\ \Delta(f) \atop \dim \sigma \leq k}(-1)^{\dim \sigma}\left(F_{\sigma}(k)+F_{\sigma}(k-1)-\binom{n-\dim\sigma}{n-k+1}\right)(\dim \sigma) ! V(\sigma),
	\end{equation*}
	where $V(\sigma)$ is the volume of $\sigma$ normalized by the assumption that a fundamental domain of the lattice $\mathbb{Z}^n\cap(\text{affine space of }\sigma)$ has unit volume and $F_{\sigma}(i)$ is the number of $i$-dimensional faces of $\Delta(f)$ that contain $\sigma$.
\end{Conj}
Take $k=5$, we compute $w_5$ for our example using Adolphson-Sperber's Conjecture.
\begin{align*}
w_5 =& (-1)^{5} \sum_{\substack{0 \in \sigma\ \text{face of}\ \Delta\\ \dim \sigma \leq 5}}
(-1)^{\dim\sigma}\left(F_{\sigma}(5)+F_{\sigma}(4)-\binom{5-\dim\sigma}{1}\right)
(\dim \sigma) ! V(\sigma)\\
=& 9F_0(5)-F_0(4)(1+1-1)+F_0(3)(1+2-2)-(6(1+3-3)+9(1+4-3))\\
&+F_0(1)(1+6-4)-F_0(0)(1+9-5)\\
=& 7 \neq 6.
\end{align*}
This result contradicts with the one we obtained.
Namely, Conjecture \ref{conj9} fails in our example.

\nocite{*}
\bibliographystyle{amsalpha}
\bibliography{ref}


\section{Appendix}
Notations in this section follow from the ones in section \ref{sec3}.
Recall that $\Delta=\Delta(g)\in \R^5$ and $\delta_i(1\leq i \leq 9)$ denote codimension 1 faces of $\Delta$ not containing the origin. Vertices of $\delta_i$ are listed as follows.
\begin{gather*}
\delta_1:\quad \left\{\V_{1}, \V_{2}, \V_{3}, \V_{4}, \V_{5}\right\}\\
\delta_2:\quad \left\{\V_{1}, \V_{2}, \V_{3}, \V_{5}, \V_{7}\right\}\\
\delta_3:\quad \left\{\V_{1}, \V_{2}, \V_{4}, \V_{5}, \V_{7}\right\}\\
\delta_4:\quad \left\{\V_{1}, \V_{3}, \V_{4}, \V_{5}, \V_{6}\right\}\\
\delta_5:\quad \left\{\V_{1}, \V_{3}, \V_{5}, \V_{6}, \V_{7}\right\}\\
\delta_6:\quad \left\{\V_{1}, \V_{4}, \V_{5}, \V_{6}, \V_{7}\right\}\\
\delta_7:\quad \left\{\V_{2}, \V_{3}, \V_{4}, \V_{5}, \V_{6}\right\}\\
\delta_8:\quad \left\{\V_{2}, \V_{3}, \V_{5}, \V_{6}, \V_{7}\right\}\\
\delta_9:\quad \left\{\V_{2}, \V_{4}, \V_{5}, \V_{6}, \V_{7}\right\}
\end{gather*}

Let $\tau_k$ be faces of $\Delta$ containing the origin with $\dim\tau_k=k$. Let $F_0(k)$ be the number of $\tau_k$. Vertices of faces $\tau_k$ are listed in Table \ref{tab5}, in which $0$ denotes the origin $(0, 0, 0, 0, 0)$.
\allowdisplaybreaks
\begin{table}[htpb]
	\begin{tabularx}{13.5cm}{cXc}
		\toprule
		$k$\quad & Vertices of $\tau_k$ & $F_0(k)$\\
		\midrule
		$0$\quad & 0 & $1$ \\
		\midrule
		$1$\quad & $\left\{0, \V_{1}\right\}, \left\{0, \V_{2}\right\}, \left\{0, \V_{3}\right\}, \left\{0, \V_{4}\right\}, \left\{0, \V_{6}\right\}, \left\{0, \V_{7}\right\}$ & $6$ \\
		\midrule
		$2$\quad & $\left\{0, \V_{1}, \V_{2}\right\},\left\{0, \V_{1}, \V_{3}\right\},\left\{0, \V_{1}, \V_{4}\right\},\left\{0, \V_{1}, \V_{6}\right\},\left\{0, \V_{1}, \V_{7}\right\},$
		$\left\{0, \V_{2}, \V_{3}\right\},\left\{0, \V_{2}, \V_{4}\right\},\left\{0, \V_{2}, \V_{6}\right\},\left\{0, \V_{2}, \V_{7}\right\},\left\{0, \V_{3}, \V_{4}\right\},$
		$\left\{0, \V_{3}, \V_{6}\right\},\left\{0, \V_{3}, \V_{7}\right\},\left\{0, \V_{4}, \V_{6}\right\},\left\{0, \V_{4}, \V_{7}\right\},\left\{0, \V_{6}, \V_{7}\right\}$ & $15$ \\
		\midrule
		$3$\quad & $\left\{0, \V_{1}, \V_{2}, \V_{3}\right\},\left\{0, \V_{1}, \V_{2}, \V_{4}\right\},\left\{0, \V_{1}, \V_{2}, \V_{7}\right\},\left\{0, \V_{1}, \V_{3}, \V_{4}\right\},$
		$\left\{0, \V_{1}, \V_{3}, \V_{6}\right\},\left\{0, \V_{1}, \V_{3}, \V_{7}\right\},\left\{0, \V_{1}, \V_{4}, \V_{6}\right\},\left\{0, \V_{1}, \V_{4}, \V_{7}\right\},$
		$\left\{0, \V_{1}, \V_{6}, \V_{7}\right\},\left\{0, \V_{2}, \V_{3}, \V_{4}\right\},\left\{0, \V_{2}, \V_{3}, \V_{6}\right\},\left\{0, \V_{2}, \V_{3}, \V_{7}\right\},$
		$\left\{0, \V_{2}, \V_{4}, \V_{6}\right\},\left\{0, \V_{2}, \V_{4}, \V_{7}\right\},\left\{0, \V_{2}, \V_{6}, \V_{7}\right\},\left\{0, \V_{3}, \V_{4}, \V_{6}\right\},$
		$\left\{0, \V_{3}, \V_{6}, \V_{7}\right\},\left\{0, \V_{4}, \V_{6}, \V_{7}\right\}$  & $18$ \\
		\midrule
		$4$\quad & $\left\{0, \V_{1}, \V_{2}, \V_{3}, \V_{4}\right\},\left\{0, \V_{1}, \V_{2}, \V_{3}, \V_{7}\right\},\left\{0, \V_{1}, \V_{2}, \V_{4}, \V_{7}\right\},$
		$\left\{0, \V_{1}, \V_{3}, \V_{4}, \V_{6}\right\}, \left\{0, \V_{1}, \V_{3}, \V_{6}, \V_{7}\right\},\left\{0, \V_{1}, \V_{4}, \V_{6}, \V_{7}\right\},$
		$\left\{0, \V_{2}, \V_{3}, \V_{4}, \V_{6}\right\},\left\{0, \V_{2}, \V_{3}, \V_{6}, \V_{7}\right\},\left\{0, \V_{2}, \V_{4}, \V_{6}, \V_{7}\right\}$  & $9$ \\
		\bottomrule
	\end{tabularx}
	\caption{Vertices of $\tau_k$}\label{tab5}
\end{table}
\end{document}